\documentclass[11pt]{article}

\usepackage[english]{babel}
\usepackage{amsmath}
\usepackage{amssymb}
\usepackage{amsfonts}
\usepackage{geometry}
\usepackage{t1enc}
\usepackage[all]{xy}
\usepackage{graphicx}
\usepackage{amstext,amsthm}
\usepackage{mathrsfs}
\usepackage{stmaryrd}

\newtheorem{thm}{Theorem}[section]

\newtheorem{lem}[thm]{Lemma}
\newtheorem{cor}[thm]{Corollary}

\theoremstyle{definition}

\newcommand{\R}{\mathbb{R}}
\newcommand{\C}{\mathbb{C}}

\newcommand{\<}{\langle}

\newcommand{\p}{\psi}
\newcommand{\f}{\varphi}

\newcommand{\g}{\mathfrak{g}}
\newcommand{\D}{\mathscr{D}}

\newcommand{\Z}{\mathbb{Z}}
\newcommand{\eval}[2][\right]{\relax
\ifx#1\right\relax \left.\fi#2#1\rvert}
\newcommand{\iso}{\xrightarrow{\ \sim\ }}

\renewcommand{\g}{\mathfrak{g}}
\renewcommand{\>}{\rangle}
\renewcommand{\a}{\mathfrak a}

\renewcommand{\to}{\longrightarrow}
\renewcommand{\H}{\mathcal H}

\DeclareMathOperator{\id}{id}

\DeclareMathOperator{\Ima}{Im}

\title{\Huge{\textbf{On Surjectivity of Invariant Differential Operators}}}
\author{Thomas Hjortgaard Danielsen}\date{}

\begin{document}

\maketitle

\begin{abstract}By proving a topological Paley-Wiener Theorem for Riemannian symmetric spaces of non-compact type, we show that a non-zero
invariant differential operator is a homeomorphism from the space of test functions
onto its image and hence surjective when extended to the space of
distributions. \end{abstract}
\vspace{0.7cm}

\section{Introduction}\label{sec:intro}
Let $\D(\R^n)$ be the space of compactly supported smooth functions on $\R^n$. Define the Euclidean Paley-Wiener space $\H^R(\C^n)$ to be the space of holomorphic maps $\f:\C^n\to\C$ such that
$$
|\f(\lambda)|\le C_N e^{R|\Ima\lambda|}(1+|\lambda|)^{-N}.
$$
This is topologized by the seminorms $\interleave \f\interleave_N$ which are the smallest constants $C_N$ for which the estimates hold. By the Paley-Wiener Theorem, the Euclidean Fourier transform
$$
\widetilde f(\xi):=\int_{\R^n}f(x)e^{-ix\cdot\xi}dx
$$
is a linear homeomorphism $\D(\R^n)\to \bigcup_{R>0}\H^R(\C^n)$ when the spaces are given the inductive limit topology.

Now let $D\neq 0$ be a constant coefficient differential operator on $\R^n$. It is well-known that conjugation of such a differential operator with the Fourier transform is just multiplication by a (non-zero) polynomial, $P_D$. This multiplication map turns out to be a homeomorphism onto its image, and consequently this holds also for $D$. An immediate consequence of this is that the differential operator is surjective on the space of distributions and hence always admit weak solutions (for more on this see \cite{55} Chapter VII).

In the article \cite{1012} it is stated that this result may be generalized to invariant differential operators on symmetric spaces of non-compact type. However the central arguments are either rather sketchy (\cite{1012}, Lemma 8) or non-existing (\cite{1012}, Theorem 7). In this article we remedy this by providing a different approach.

\section{Notation}
First we introduce some notation. Let $X$ be a Riemannian symmetric space of non-compact type, i.e. a quotient $G/K$ where $G$ is a semisimple non-compact Lie-group with finite center and $K$ is a maximal compact subgroup. The Lie algebra $\g$ has a Cartan decomposition $\g=\mathfrak k\oplus\mathfrak p$ and we pick in $\mathfrak p$ a maximally abelian subalgebra $\a$. Define $M:=Z_K(\a)$ and $B:=K/M$ and let $\a_\C^*$ denote the complexification of the dual of $\a$. Let $\Sigma\subseteq \a^*$ denote the set of restricted roots w.r.t. $(\g,\a)$ and let $\Sigma^+$ denote the set of positive roots relative to a fixed Weyl chamber. Put $\rho:=\frac12\sum_{\lambda\in\Sigma^+}\lambda$.

We let $\D(X)$ denote the set of test functions ($C^\infty$-functions with compact support) equipped with its usual inductive limit topology, and let $\D_R(X)$ denote the closed subspace of test functions whose support is contained in the closed $R$-ball $\overline{B_R(eK)}$ around $eK\in X$.

The Fourier transform $\widetilde f:\a_\C^*\times B\to \C$ of a function $f\in \D(X)$ is defined by (see e.g. \cite{47} Chapter III)
$$
\widetilde f(\lambda,b):=\int_X f(x)e^{(-i\lambda+\rho)A(x,b)}dx
$$
where $A(gK,kM):=A(k^{-1}g)$ with $A:G\to \a$ being the Iwasawa projection from the $NAK$-decomposition of $G$.

Let $\H^R(\a_\C^*\times B)$ denote the space of smooth functions $\p:\a_\C^*\times B\to \C$ which are holomorphic on $\a_\C^*$ and which satisfy the following growth condition
\begin{equation}\label{eq:growthcond1}
\forall N\in \Z_{>0}\:\exists C_N\:\forall\lambda,b:\ |\p(\lambda,b)|\le C_N e^{R|\Ima\lambda|}(1+|\lambda|)^{-N}.
\end{equation}

Furthermore, we denote by $\H^R(\a_\C^*\times B)_W$ the subset of functions satisfying the following Weyl invariance for each $s\in W$:
\begin{equation}\label{eq:Weylinv}
\forall x\in X\,\forall\lambda\in\a_\C^*:\quad \int_B e^{(is\cdot\lambda+\rho)A(x,b)}\p(s\cdot\lambda,b)db=\int_B e^{(i\lambda+\rho)A(x,b)}\p(\lambda,b)db.
\end{equation}

\section{A Topological Paley-Wiener Theorem}
First we need a topological Paley-Wiener Theorem and in order to do so, we have to topologize the space $\H^R(\a_\C^*\times B)_W$. For this, introduce the space $\H^R(\a_\C^*, L^2(B))$ to be the space consisting of holomorphic maps $\p:\a_\C^*\to L^2(B)$ satisfying
$$
\|\p(\lambda)\|_{L^2(B)}\le C_N e^{R|\Ima\lambda|}(1+|\lambda|)^{-N}
$$
for all $N$. We define $\interleave\p\interleave_N$ to be the smallest such constant $C_N$, and we topologize $\H^R(\a_\C^*, L^2(B))$ by this family of seminorms. This turns it into a Fréchet space. The Weyl invariance \eqref{eq:Weylinv} still makes sense in this generalized setting, and thus we define $\H^R(\a_\C^*, L^2(B))_W$ to be the subset of Weyl invariant elements. This is a closed subspace and hence a Fréchet space.

We have an obvious inclusion $\H^R(\a_\C^*\times B)_W\to \H^R(\a_\C^*, L^2(B))_W$ and this inclusion turns out to be surjective:

\begin{lem}
It holds that $\H^R(\a_\C^*, L^2(B))_W=\H^R(\a_\C^*\times B)_W$ as vector spaces.
\end{lem}
\begin{proof}
For $\p\in \H^R(\a_\C^*, L^2(B))_W$ define
$$
f(x):=\int_{\a^*\times B}\p(\lambda,b)e^{(i\lambda+\rho)A(x,b)}|c(\lambda)|^{-2}dbd\lambda=\int_{\a^*}\<\p(\lambda),e^{(-i\lambda+\rho)A(x,\cdot)}\>_{L^2(B)}|c(\lambda)|^{-2}d\lambda.
$$
Obviously, $f$ is a smooth function. By examining the proof of bijectivity of $\mathcal F:\D(X)\to \H(\a_\C^*\times B)_W$ (e.g. in \cite{47} p. 278--280), it is seen that $f$ is supported in the closed $R$-ball $B_R(eK)$ and that $\widetilde f-\p=0$ almost everywhere, and thus $f$ is a smooth representative of $\f$. Furthermore $\widetilde f$ satisfies the stronger growth condition \eqref{eq:growthcond1} and hence $f\in\H^R(\a_\C^*\times B)_W$.
\end{proof}

Now $\H^R(\a_\C^*\times B)_W$ inherits the topology from $\H^R(\a_\C^*, L^2(B))_W$ (given by the seminorms $\interleave\cdot\interleave_N$), and hence it becomes a Fréchet space. Furthermore we define
$$
\H(\a_\C^*\times B)_W:=\bigcup_{R\in \Z_{>0}}\H^R(\a_\C^*\times B)_W
$$
and give it the inductive limit topology.

\begin{thm}[\textbf{Topological Paley-Wiener}]
The Fourier transform $\mathcal F:\D(X)\to \H(\a_\C^*\times B)_W$ is a linear homeomorphism. Furthermore $\widetilde f\in \H^R(\a_\C^*\times B)_W$ if and only if $f\in \D_R(X)$.
\end{thm}
\begin{proof}
The bijectivity of $\mathcal F$ as well as the last claim is stated
and proved in \cite{Helg1} Theorem III.5.1.

Now we consider $\mathcal F:\D_R(X)\to \H^R(\a_\C^*\times B)_W$ for a given $R$. For $f\in \D_R(X)$ it is
straightforward to check the inequality for each $N$:
$$
\interleave\mathcal F f\interleave_N\le
C\sup_{\lambda\in\a_\C^*,b\in B}
e^{-R|\Ima\lambda|}\int_{\overline{B_R(eK)}}|Df(x)||e^{(-i\lambda+\rho)A(x,b)}|dx
$$
where $D$ is the invariant differential operator (of order $2N$) on $X$ corresponding to the invariant polynomial $(1+|\lambda|^2)^N$ (as in \eqref{eq:diag}) and where $C$ is a constant depending on $N$ and $R$. Since $x\in
\overline{B_R(eK)}$ we have by \cite{Helg2} p. 476 eq. (13) that
$|A(x,b)|\le R$ and hence we see that
$$
e^{-R|\Ima\lambda|}|e^{(-i\lambda+\rho)A(x,b)}|=e^{(\Ima\lambda+\rho)A(x,b)-R|\Ima\lambda|}\le
e^{\rho(A(x,b))}\le e^{R|\rho|}.
$$
Hence we get
$\interleave\mathcal Ff\interleave_N\le C\|f\|_{2N}$, where
$\|\cdot\|_{2N}$ is one of the standard seminorms on $\D_R(X)$, i.e.
$\mathcal F:\D_R(X)\to \H^R(\a_\C^*\times B)_W$ is continuous.

Thus the Fourier transform is a homeomorphism $\D_R(X)\iso\H^R(\a_\C^*\times B)_W$
since these spaces are Fréchet. Hence it is also a
homeomorphism when defined on $\D(X)$.
\end{proof}

\section{Consequences of the Paley-Wiener Theorem}

Now, let $D$ be a non-zero invariant differential operator. There exists a $W$-invariant polynomial $P_D\neq 0$ on $\a_\C^*$ (this is a consequence of \cite{45} Theorem II.4.6 and Lemma II.5.14 where we identify $\mathbb D(A)$ with $W$-invariant polynomials on $\a$) such that the following diagram commutes
\begin{equation}\label{eq:diag}
\xymatrix{\D_R(X)\ar[rr]^-\sim \ar[d]_D & &\H^R(\a_\C^*\times B)_W\ar[d]^{M_{P_{D}}}\\
\D_R(X)\ar[rr]^-\sim & & \H^R(\a_\C^*\times B)_W}
\end{equation}
where $M_{P_D}$ is multiplication by $P_D$. The first goal is to show that $M_{P_D}$ and hence $D$ are linear homeomorphisms onto their images. Injectivity of $M_{P_D}$ is clear by holomorphicity since $P_D\neq 0$.

Another payoff of considering $\H^R(\a_\C^*, L^2(B))$ rather than $\H^R(\a_\C^*\times B)$, is that
it admits a nice description as a tensor product. First, however, note that the spaces we defined in Section \ref{sec:intro} actually make sense for any finite-dimensional inner product space $V$ and its complexification $V_\C$. Thus we can define $\D(V)$, $\D_R(V)$, $\H^R(V_\C)$ and so on, as well as a Euclidean Fourier transform which will be a homeomorphism $\mathcal F:\D(V)\iso \H(V_\C)$. In the following we will take $V=\a^*$ with the Killing form as inner product.

Let $\H^R(\a_\C^*)\widehat\otimes L^2(B)$ denote the completion of the
algebraic tensor product in either the projective or injective
topology (they are both equal since $\H^R(\a_\C^*)\cong \D_R(\a^*)$ is
nuclear). Then

\begin{lem}
There exists a natural linear homeomorphism
$\H^R(\a_\C^*)\widehat\otimes L^2(B)\iso\H^R(\a_\C^*, L^2(B))$.
\end{lem}
\begin{proof}
Letting $\D_R(\a^*,L^2(B))$
denote the space of smooth $L^2(B)$-valued functions on $\a^*$ with
support in the $R$-ball, we define a Fourier transform
$$
\mathcal F:\D_R(\a^*,L^2(B))\to \H^R(\a_\C^*, L^2(B))
$$
by
$$
(\mathcal F f)(\lambda):=\int_{\a^*} f(x)e^{-i\<x,\lambda\>}dx
$$
($L^2(B)$-valued integration). For all $v\in L^2(B)$ it holds that $\<\mathcal
Ff(\lambda),v\>=\mathcal F(\<f,v\>)(\lambda)$. Since the function
$x\longmapsto\<f(x),v\>$ is an element of $\D_R(\a^*)$ we see that
$\lambda\longmapsto\<\mathcal Ff(\lambda),v\>$ is holomorphic, i.e.
$\mathcal Ff$ is weakly holomorphic, hence holomorphic (as it takes
values in a Hilbert space). Furthermore we note that $\mathcal F$ has
an inverse:
$$
(\mathcal
F^{-1}\p)(x)=\int_{\a^*}\p(\lambda)e^{i\<x,\lambda\>}d\lambda.
$$
Continuity of $\mathcal F$ is easily checked and since the spaces in question are both Fréchet, $\mathcal F$ is a linear homeomorphism. The lemma
now follows from the fact that
$\D_R(\a^*,L^2(B))\cong\D_R(\a^*)\widehat\otimes L^2(B)$ (which is a consequence of \cite{52} Theorem 44.1) and that
$\D_R(\a^*)\cong \H^R(\a^*_\C)$ by the Euclidean Fourier transform.
\end{proof}

Now returning to the commuting diagram \eqref{eq:diag} we see that under the
identification of $\H^R(\a_\C^*, L^2(B))$ with
$\H^R(\a_\C^*)\widehat\otimes L^2(B)$, the map
$M_{P_D}:\H^R(\a_\C^*, L^2(B))\to \H^R(\a_\C^*, L^2(B))$ is
replaced by $M_{P_D}\otimes \id_{L^2(B)}$. And from the Euclidean
theory we know that $M_{P_D}:\H^R(\a_\C^*)\to \H^R(\a_\C^*)$ is a
homeomorphism onto its image, hence the same holds for
$M_{P_D}\otimes \id_{L^2(B)}$ (cf. \cite{52} Proposition 43.7). By restriction to $\H^R(\a_\C^*\times
B)_W$ we get:

\begin{lem}
The multiplication map $M_{P_D}:\H^R(\a_\C^*\times B)_W\to
\H^R(\a_\C^*\times B)_W$ is a homeomorphism onto its image.
\end{lem}

Finally, we need to transfer this to the inductive limit. For this we
need the following lemma which is an immediate generalization of \cite{47} Lemma III.5.13 to functions with Hilbert space
values:

\begin{lem}\label{lem:pollemma}
Assume $P:\a_\C^*\to \C$ is a nonzero polynomial and that
$\p:\a_\C^*\to L^2(B)$ is a holomorphic function such that $P\p\in
\H^R(\a_\C^*, L^2(B))$, then $\p\in \H^R(\a_\C^*, L^2(B))$.
\end{lem}

Restricting to $\H^R(\a_\C^*, L^2(B))_W=\H^R(\a_\C^*\times B)_W$ and referring to the commutative diagram \eqref{eq:diag} we get that if $f\in\D(X)$ is such that $Df\in\D_R(X)$, then $f\in
\D_R(X)$.

We now arrive at our main theorem:

\begin{thm}
Let $D\neq 0$ be an invariant differential operator, then
$D:\D(X)\to \D(X)$ is a homeomorphism onto its image.
\end{thm}
\begin{proof}
Since $D:\D_R(X)\to \D_R(X)$ is a homeomorphism onto its image,
it is clear that $D$ is injective and continuous. We just need to
show that it's open. For this, let $U\subseteq \D(X)$ be open, i.e.
$U\cap D_R(X)$ is open for all $R$. We need to show that $D(U)\cap
\D_R(X)$ is open for all $R$. Claim: $D(U)\cap \D_R(X)=D(U\cap
\D_R(X))$. The inclusion ``$\supseteq$'' is clear, since $D$
decreases support. The inclusion ``$\subseteq$'' is a consequence of
Lemma \ref{lem:pollemma}. But the right hand side $D(U\cap \D_R(X))$
is open in $\D_R(X)$, and thus the result follows.
\end{proof}

Let $\D'(X)$ denote the set of distributions over $X$, i.e. the dual of $\D(X)$. A simple Hahn-Banach argument (see e.g. \cite{55} p. 236) yields
\begin{cor}
$D:\D'(X)\to \D'(X)$ is surjective.
\end{cor}

\end{document}